\documentclass[12pt]{article}

\usepackage{amsmath}
\usepackage{graphicx}

\usepackage{amsthm}
\usepackage{cite}

\theoremstyle{definition}
\newtheorem{thm}{Theorem}[section]
\newtheorem{lem}[thm]{Lemma}

\newtheorem{claim}[thm]{Claim}

\newcommand{\dist}{\ensuremath{\mathrm{dist}}}

\title{An Upper Bound for the Double Domination Number in Maximal Outerplanar Graphs}
\author{Toru Araki}
\date{Faculty of Informatics, Gunma University \\ Maebashi, Gunma, Japan}

\begin{document}
\maketitle

\begin{abstract}
  In a graph $G$, a vertex dominates itself and its neighbors.
  A subset $S$ of vertices of $G$ is a double dominating set of $G$ if
  every vertex is dominated by at least two vertices in $S$.
  The double domination number $\gamma_{\times 2}(G)$ of $G$ is the
  minimum cardinality of a double dominating set of $G$.
  In this paper, we prove that, for a maximal outerplanar graph $G$,
  the double domination number $\gamma_{\times 2}(G)$ is at most
  $(n+k)/2$, where $k$ is the number of pairs of consecutive vertices
  on the outer cycle but at distance at least 3.
  Although this bound was previously proposed by Abd Aziz, Rad and
  Kamarulhaili (A note on the double domination number in maximal
  outerplanar and planar graphs, RAIRO Operations Research, 56 (2022)
  3367--3371), their proof was found to be incomplete.
  In this paper we establish the validity of this result by providing
  a complete proof.
\end{abstract}

\section{Introduction}
\label{sec:introduction}

All graphs considered in this paper are finite, simple, and
undirected.
For a graph $G$, $V(G)$ and $E(G)$ are the sets of vertices and edges
of $G$, respectively.
The distance between vertices $u$ and $v$ is the minimum length of a
path between $u$ and $v$, and is denoted by $\dist_{G}(u, v)$.
For nonadjacent vertices $u$ and $v$, $G+uv$ is the graph obtained by
adding an edge $uv$.
For a set $S$ of vertices, $G-S$ is the graph obtained by deleting
vertices in $S$ and their incident edges.
The \emph{open neighborhood} $N_{G}(v)$ of a vertex $v$ is the set of
vertices that are adjacent to $v$, while the \emph{closed
  neighborhood} of $v$ is the set $N_{G}[v] = N_{G}(v) \cup \{v\}$.
The degree of $v$, denoted by $\deg_{G} v$, is defined by $\deg_{G} v
= |N_{G}(v)|$.
A subset $S \subseteq V(G)$ is a \emph{dominating set} of $G$ if, for 
any vertex $v$ in $V(G) \setminus S$, it holds $|N_{G}[v] \cap S| \geq 1$.
The \emph{domination number} $\gamma(G)$ is the minimum cardinality of
a dominating set of $G$.
A subset $S \subseteq V(G)$ is a \emph{double dominating set} of $G$
if, for any vertex $v$ in $V(G) \setminus S$, it holds $|N_{G}[v] \cap S| \geq 2$.
The \emph{double domination number} $\gamma_{\times 2}(G)$ is the
minimum cardinality of a double dominating set of $G$.
For a comprehensive treatment of domination in graphs, the reader is
referred to the
books~\cite{haynes21:_domin_graph_core,haynes20:_topic_domin_graph,haynes21:_struc_domin_graph,haynes98:_domin_graph,haynes98b}.

Harary and Haynes~\cite{harary00} defined $k$-tuple domination which
is a generalization of domination.
A subset $S \subseteq V(G)$ is a \emph{$k$-tuple dominating set} of $G$
if, for any vertex $v$ in $V(G) \setminus S$, it holds $|N_{G}[v] \cap S| \geq k$.
When $k=1$, a $k$-tuple dominating set is a dominating set.
A $k$-tuple dominating set where $k=2$ is called a \emph{double
  dominating set}.
The concept of double domination in graphs was studied in
\cite{henning05:_graph,blidia06,harant05,hajian19}.
Henning~\cite{henning05:_graph} showed that $\gamma_{\times 2}(G) \leq
3n/4$ for any graph $G$ which is not a cycle of 5 vertices.

A \emph{plane embedding} of a planar graph $G$ is an embedding of $G$
in a plane such that the edges of $G$ do not intersect each other.
A planar graph with a plane embedding is called a \emph{plane graph}.
A graph $G$ is \emph{outerplanar} if it has an embedding in the plane
such that all vertices belong to the boundary of its outer face (the
unbounded face).
An outerplanar graph $G$ is \emph{maximal} if $G+uv$ is not
outerplanar for any two nonadjacent vertices $u$ and $v$.
A maximal outerplanar graph has an embedding such that all inner faces
are triangles.

Domination in maximal outerplanar graphs has been extensively
studied~\cite{matheson96:_domin,campos13,tokunaga13:_domin}.
Matheson and Tarjan~\cite{matheson96:_domin} proved a tight upper bound for
the domination number on the class of \emph{triangulated discs}:
graphs that have an embedding in the plane such that all of their
faces are triangles, except possibly one.  They proved that
$\gamma(G) \leq n/3$ for any $n$-vertex triangulated disc, and also
showed that this bound is tight.  For maximal outerplanar graphs,
better upper bounds are obtained.  Campos and
Wakabayashi~\cite{campos13} showed that if $G$ is a maximal
outerplanar graph of $n$ vertices, then $\gamma(G) \leq (n+k)/4$ where
$k$ is the number of vertices of degree 2.  Tokunaga proved the same
result independently in~\cite{tokunaga13:_domin}.  For results on
other types of domination in maximal outerplanar graphs, we refer the
reader
to~\cite{aita24,araki18,dorfling17:_total,lemanska17:_total,lemanska19:_convex,henning25:_dijun}.

Zhuang~\cite{zhuang22:_doubl} studied double domination in maximal
outerplanar graphs and obtained the following results.

\begin{thm}[\cite{zhuang22:_doubl}]
  For a maximal outerplanar graph $G$ of $n \geq 3$ vertices,
  $\gamma_{\times 2}(G) \leq 2n/3$.
\end{thm}

\begin{thm}[\cite{zhuang22:_doubl}]
  \label{thm:zhuang}
  For a maximal outerplanar graph $G$ of $n \geq 3$ vertices,
  $\gamma_{\times 2}(G) \leq (n+t)/2$, where $t$ is the number of
  vertices of degree 2.
\end{thm}

Recently, Abd~Aziz, Rad and Kamarulhaili~\cite{aziz22} improved the
bound of Theorem~\ref{thm:zhuang}.
In this paper, we point out a flaw of the proof in~\cite{aziz22} and
provide a complete proof for the upper bound for the double domination
in maximal outerplanar graphs.
Our proof is inspired by the work of Henning et
al.~\cite{henning25:_dijun}.

\section{Preliminaries}
\label{sec:preliminaries}

A maximal outerplanar graph $G$ can be embedded in the plane such that
the boundary of the outer face is a Hamiltonian cycle and each inner
face is a triangle.
A maximal outerplanar graph embedded in the plane is called a
\emph{maximal outerplane graph}.
For such an embedding of $G$, we denote by $C_{G}$ the Hamiltonian
cycle which is the boundary of the outer face.
We refer to inner faces of a maximal outerplane graph as
\emph{triangles}.
An inner face of a maximal outerplane graph $G$ is an \emph{internal
  triangle} if it is not adjacent to the outer face.
Two triangles are \emph{adjacent} if they share a common edge.
The \emph{dual tree} $T$ of a maximal outerplane graph $G$ is the
graph whose vertices correspond to the triangles of $G$, and where two
vertices of $T$ are adjacent if and only if their corresponding
triangles of $G$ are adjacent.
The dual tree $T$ has maximum degree at most 3, and a vertex of degree
3 corresponds to an internal triangle of $G$, and a vertex of degree 1
corresponds to a triangle that contains a vertex of degree 2 in $G$.

Let $v_{1}, v_{2}, \dots, v_{t}$ be all the vertices of degree 2 which
appear in the clockwise direction on $C_{G}$.
A vertex $v_{i}$ is called a \emph{bad} vertex if the clockwise path
from $v_{i}$ to $v_{i+1}$ on the outer cycle $C_{G}$ has at least 3
edges for $i = 1, 2, \dots, t$, where the subscripts are taken modulo
$t$.

In \cite{aziz22}, the following theorem was provided.
\begin{thm}[Theorem~2.1 in~\cite{aziz22}]
  Let $G$ be a maximal outerplane graph of $n \geq 4$ vertices.
  If $G$ has $k$ bad vertices, then $\gamma_{\times 2}(G) \leq
  (n+k)/2$.
\end{thm}
This theorem was proved by mathematical induction and case-by-case
analysis.
Let $G_{1} = G - \{v_{1},v_{2},\dots,v_{t}\}$.
Then $G_{1}$ is a maximal outerplane graph.
Let $u$ be a vertex of degree 2 in $G_{1}$.
Since the degree of $u$ is not 2 in $G$, it is adjacent to some vertex
$v \in \{v_{1},v_{2},\dots,v_{t}\}$.
Let $v$ be adjacent to $u$ and $u_{1}$, and let $u_{0}$ be the vertex
just before $u_{1}$ and $u_{2}$ be the vertex just after $u$.
The proof in~\cite{aziz22} overlooks the case represented by
Fig.~\ref{fig:lack}, that is, $\deg_{G} u_{1} = 4$ and $u_{0}u_{2} \in
E(G)$.
In Section~2 of~\cite{aziz22}, only the cases where
$\deg_{G} u_{1} = 3$ and $\deg_{G} u_{1} \geq 5$ were considered.
Hence, the proof in~\cite{aziz22} is incomplete.

\begin{figure}[tbp]
  \centering
  \includegraphics{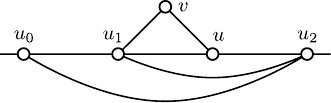}
  \caption{The situation that was missing from the proof.}
  \label{fig:lack}
\end{figure}

\section{Main Results}
\label{sec:results}

In this section, we give a proof of the following theorem.

\begin{thm}
  \label{thm:main}
  For any maximal outerplane graph $G$ with $n \geq 4$ vertices and
  $k$ bad vertices,
  \begin{displaymath}
    \gamma_{\times 2}(G) \leq \frac{n + k}{2}.
  \end{displaymath}
\end{thm}

Throughout this section, we suppose that a double dominating set $S$
does not contain vertices of degree 2.
That is, a maximal outerplane graph $G$ has a double dominating set
$S$ such that $S$ does not contain vertices of degree 2 and $|S| \leq
(n+k)/2$.
Note that, if $u$ is a vertex of degree 2, the two vertices adjacent
to $u$ are contained in $S$.
We can easily show that Theorem~\ref{thm:main} is true when $n$ is small.

\begin{lem}
  \label{lem:small}
  If $4 \leq n \leq 8$, then $\gamma_{\times 2}(G) \leq (n+k)/2$ holds.
\end{lem}

Suppose to the contrary that there exists a counterexample to
Theorem~\ref{thm:main}.
Let $G$ be a counterexample with the minimum number $n$ of vertices,
and let $G$ have $k$ bad vertices.
Since $G$ is a counterexample of minimum number of vertices, it
satisfies $\gamma_{\times 2}(G) > (n+k)/2$.
Furthermore, if $G'$ is a maximal outerplane graph with $n'$ vertices
for $4 \leq n' < n$ and with $k'$ bad vertices, then
$\gamma_{\times 2}(G) \leq (n'+k')/2$ holds.

\begin{lem}
  \label{lem:dist}
  Let $T$ be the dual tree of $G$ and $t_{1}$ be a leaf of $T$.
  Then the following properties hold.
  \begin{itemize}
  \item[(1)] $T$ is not a path, that is, $T$ has a vertex of degree 3.

  \item[(2)] If $t$ is the nearest vertex of degree 3 from $t_{1}$ in
    $T$, then $\dist_{T}(t_{1}, t) \in \{1, 2, 4, 6\}$.

  \item[(3)] If $t$ is the nearest vertex of degree 3 from $t_{1}$ in
    $T$ and $\dist_{T}(t_{1}, t)=4$, then the subgraph of $G$ induced
    by the vertices in the triangles corresponding to the vertices on
    the path between $t$ and $t_{1}$ in $T$ is isomorphic to the graph
    in Fig.~\ref{fig:lem}(c).

  \item[(4)] If $t$ is the nearest vertex of degree 3 from $t_{1}$ in
    $T$ and $\dist_{T}(t_{1}, t)=6$, then the subgraph of $G$ induced
    by the vertices in the triangles corresponding to the vertices on
    the path between $t$ and $t_{1}$ in $T$ is isomorphic to the graph
    in Fig.~\ref{fig:lem}(d).
  \end{itemize}
\end{lem}

\begin{figure}[tbp]
  \centering
  \includegraphics[width=\textwidth]{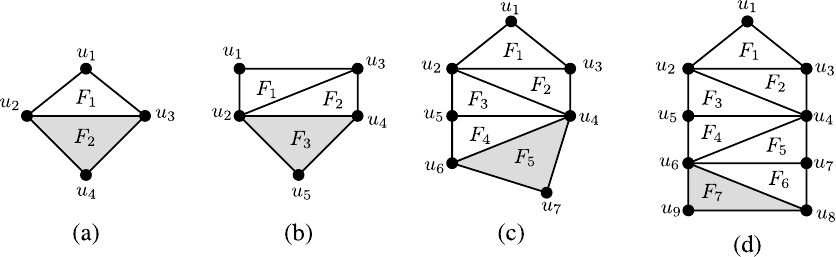}
  \caption{Lemma~\ref{lem:dist}.  Shaded triangles are internal
    triangles.}
  \label{fig:lem}
\end{figure}

\begin{proof}
  Let $t_{1}$ be a leaf of $T$ and $F_{1}$ be the triangle
  corresponding to $t_{1}$.  Let $V(F_{1}) = \{u_{1}, u_{2}, u_{3}\}$
  and $\deg_{G} u_{1} = 2$.  Let $t_{2}$ be the vertex adjacent to
  $t_{1}$ in $T$, and let $F_{2}$ be the triangle corresponding to
  $t_{2}$.
  Let $V(F_{2}) = \{u_{2}, u_{3}, u_{4}\}$.
  If $\deg_{T} t_{2} = 3$, then
  we have $t=t_{2}$ and $\dist_{T}(t_{1}, t) = 1$ and hence
  the property~(2) holds (Fig.~\ref{fig:lem}(a)).  Since $G$
  has at least 9 vertices, we have $\deg_{T} t_{2} \neq 1$.  Hence we
  assume that $\deg_{T} t_{2} = 2$.

  Let $t_{3} (\neq t_{1})$ be the vertex adjacent to $t_{2}$, and let
  $F_{3}$ be the triangle corresponding to $t_{3}$.  If
  $\deg_{T} t_{3} = 3$, then we have $t=t_{3}$ and
  $\dist_{T}(t_{1}, t) = 2$ and hence the property~(2) holds
  (Fig.~\ref{fig:lem}(b)).  Since $G$ has at least 9 vertices, we have
  $\deg_{T} t_{2} \neq 1$.  Let $u_{5}$ be the vertex in $F_{3}$ that
  is not in $F_{2}$.  We may assume that
  $V(F_{3}) = \{u_{2}, u_{4}, u_{5}\}$.

  We assume that $\deg_{T} t_{3} = 2$.  Let $t_{4} (\neq t_{2})$ be
  the vertex adjacent to $t_{3}$, and let $F_{4}$ be the triangle
  corresponding to $t_{4}$.  Let $u_{6}$ be the vertex in $F_{4}$ that
  is not in $F_{3}$.  Since $G$ has at least 9 vertices, we have
  $\deg_{T} t_{4} \neq 1$.

  \begin{figure}[tbp]
    \centering
    \includegraphics{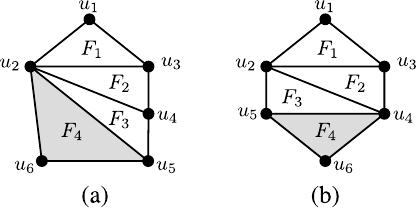}
    \caption{Claim~\ref{clm:2}.  Possible triangles adjacent to
      $F_{3}$.  Shaded triangles are internal triangles.}
    \label{fig:clm2}
  \end{figure}

  \begin{claim}
    \label{clm:2}
    $\deg_{T} t_{4} = 2$.
  \end{claim}
  \begin{proof}[Proof of Claim~\ref{clm:2}]
    Suppose to the contrary that $\deg_{T} t_{4} = 3$.
    Note that $u_{1}$ is a bad vertex.
    We consider two cases.

    Case~1: $V(F_{4})=\{u_{2}, u_{5}, u_{6}\}$ (See
    Fig.~\ref{fig:clm2}(a)).
    In this case, let $G' = G - \{u_{3}, u_{4}\} + u_{1}u_{5}$.
    Then, $G'$ has $n' = n-2$ vertices.
    Let $v$ be the vertex just after $u_{5}$ in the clockwise order on
    $C$.
    Note that $G'$ has $k' = k-1$ bad vertices if $\deg_{G} v = 2$,
    and $k' = k$ bad vertices if $\deg_{G} v > 2$.

    By the minimality of $G$, a minimum double dominating set $S'$
    of $G'$ satisfies $|S'| \leq (n'+k')/2 \leq (n+k-2)/2$,
    and $S'$ contains $u_{2}$ and $u_{5}$ since $\deg_{G'} v = 2$.
    The set $S = S' \cup \{u_{3}\}$ is a double dominating set of $G$
    and satisfies $\gamma_{\times 2}(G) \leq |S| \leq |S'| + 1 =
    (n+k)/2$, a contradiction.

    \vspace{1em}
    Case~2: $V(F_{4}) = \{u_{4}, u_{5}, u_{6}\}$ (See Fig.~\ref{fig:clm2}(b)).
    Let $G' = G - \{u_{1}, u_{2}, u_{3}\}$.
    $G'$ has $n' = n-3$ vertices and $k'=k-1$ bad vertices.
    By the minimality of $G$, a minimum double dominating set $S'$
    of $G'$ satisfies $|S'| \leq (n'+k')/2 = (n+k-4)/2$.
    The set $S = S' \cup \{u_{2}, u_{3}\}$ is a double dominating set
    of $G$ and satisfies
    $\gamma_{\times 2}(G) \leq |S| = |S'| + 2 = (n+k)/2$, a contradiction.

    Hence we have $\deg_{T} t_{4} \neq 3$, which implies that
    $\deg_{T} t_{4} = 2$.
  \end{proof}

  By Claim~\ref{clm:2}, we obtain $\dist_{T}(t, t_{1}) \neq 3$.
  \begin{claim}
    \label{clm:3}
    $V(F_{4}) = \{u_{4}, u_{5}, u_{6}\}$.
  \end{claim}
  \begin{proof}[Proof of Claim~\ref{clm:3}]
    If $V(F_{4}) = \{u_{2}, u_{5}, u_{6}\}$, then we obtain a
    contradiction similarly to Case~1 in the proof of
    Claim~\ref{clm:2}.
  \end{proof}

  Let $t_{5} (\neq t_{3})$ be the vertex adjacent to $t_{4}$, and let
  $F_{5}$ be the triangle corresponding to $t_{5}$.  Let $u_{7}$ be
  the vertex in $F_{5}$ that is not in $F_{4}$.  Since $G$ has at
  least 9 vertices, we have $\deg_{T} t_{5} \neq 1$.
  If $\deg_{T} t_{5} = 3$, then we have $t=t_{5}$ and
  $\dist_{T}(t_{1}, t) = 4$ and hence the property~(2) holds.

  \begin{figure}[tbp]
    \centering
    \includegraphics{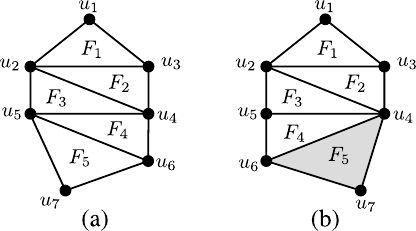}
    \caption{Claim~\ref{clm:4}.  Possible triangles adjacent to $F_{4}$
      when (a) $F_{5} = \{u_{5},u_{6},u_{7}\}$ and (b)
      $F_{5} = \{u_{4},u_{6},u_{7}\}$.}
    \label{fig:clm4}
  \end{figure}

  \begin{claim}
    \label{clm:4}
    $V(F_{5}) = \{u_{4}, u_{6}, u_{7}\}$.
  \end{claim}
  \begin{proof}[Proof of Claim~\ref{clm:4}]
    By Claim~\ref{clm:2} and \ref{clm:3}, $V(F_{5})$ is either
    $\{u_{4}, u_{6}, u_{7}\}$ or $\{u_{5}, u_{6}, u_{7}\}$.
    Suppose to the contrary that $V(F_{5}) = \{u_{5}, u_{6}, u_{7}\}$
    (See Fig.~\ref{fig:clm4}(a)).
    Let $G' = G - \{u_{1}, u_{2}, u_{3}, u_{4}\}$.
    $G'$ has $n' = n-4$ vertices and $k'$ bad vertices with
    $k' \leq k$.
    By the minimality of $G$, a minimum double dominating set $S'$
    of $G'$ satisfies $|S'| \leq (n'+k')/2 = (n+k-4)/2$.
    The set $S = S' \cup \{u_{2}, u_{3}\}$ is a double dominating set
    of $G$ and satisfies $\gamma_{\times 2}(G) \leq |S| = |S'| + 2 =
    (n+k)/2$, a contradiction.

    Hence $V(F_{5}) = \{u_{4}, u_{6}, u_{7}\}$.
  \end{proof}

  By Claim~\ref{clm:4}, if $\deg_{T} t_{5} = 3$, then the subgraph of
  $G$ induced by $\{u_{1}, u_{2}, \dots, u_{7}\}$ is isomorphic to the
  graph in Fig.~\ref{fig:lem}(c).
  Hence the property~(3) holds.

  We assume that $\deg_{T} t_{5} = 2$.
  Let $t_{6} (\neq t_{4})$ be the vertex adjacent to $t_{5}$, and let
  $F_{6}$ be the triangle corresponding to $t_{6}$.  Let $u_{8}$ be
  the vertex in $F_{6}$ that is not in $F_{5}$.  Since $G$ has at
  least 9 vertices, we have $\deg_{T} t_{6} \neq 1$.

  \begin{claim}
    \label{clm:5}
    $\deg_{T} t_{6} = 2$ and $V(F_{6}) = \{u_{6}, u_{7}, u_{8}\}$.
  \end{claim}
  \begin{proof}[Proof of Claim~\ref{clm:5}]
    Suppose to the contrary that $\deg_{T} t_{6} = 3$.
    By Claim~\ref{clm:4}, $V(F_{6})$ is either $\{u_{6}, u_{7},
    u_{8}\}$ or $V(F_{6}) = \{u_{4}, u_{7}, u_{8}\}$.
    See Fig.~\ref{fig:clm5}.
    Note that $u_{1}$ is a bad vertex.

    First assume that $V(F_{6}) = \{u_{6}, u_{7}, u_{8}\}$.  Let
    $G' = G - \{u_{1}, u_{2}, u_{3}, u_{4}, u_{5}\}$.  $G'$ has
    $n' = n-5$ vertices and $k' = k-1$ bad vertices.
    By the
    minimality of $G$, a minimum double dominating set $S'$ of $G'$
    satisfies $|S'| \leq (n'+k')/2 = (n+k-6)/2$.
    The set $S = S' \cup \{u_{2}, u_{3}, u_{4}\}$ is a double
    dominating set of $G$ and satisfies $\gamma_{\times 2}(G) \leq |S| =
    |S'| + 3 = (n+k)/2$, a contradiction.

    Next assume that $V(F_{6}) = \{u_{4}, u_{7}, u_{8}\}$.  Let
    $G' = G - \{u_{1}, u_{2}, u_{3}, u_{5}\}$.  $G'$ has $n' = n-4$
    vertices and $k' \leq k$ bad vertices.  Thus, a minimum double
    dominating set $S'$ of $G'$ satisfies
    $|S'| \leq (n'+k')/2 = (n+k-4)/2$,
    and $S'$ contains $u_{4}$ and $u_{7}$ since
    $\deg_{G'} u_{6} = 2$.  The set $S = S' \cup \{u_{2}, u_{3}\}$ is
    a double dominating set of $G$ and satisfies
    $\gamma_{\times 2}(G) \leq |S| = |S'| + 2 = (n+k)/2$, a contradiction.

    Hence we have $\deg_{T} t_{6} \neq 3$, and thus
    $\deg_{T} t_{6} = 2$.
    Furthermore, if $V(F_{6}) = \{u_{4}, u_{7}, u_{8}\}$ and
    $\deg_{T} t_{6} = 2$, then we obtain a contradiction by the similar
    argument.
    Hence we obtain $V(F_{6}) = \{u_{6}, u_{7}, u_{8}\}$.
  \end{proof}

  \begin{figure}[tbp]
    \centering
    \includegraphics{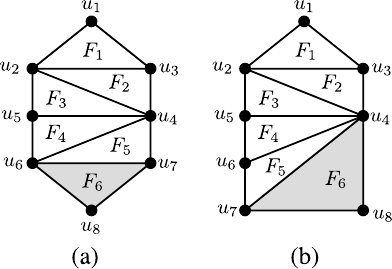}
    \caption{Claim~\ref{clm:5}.  Possible triangles adjacent to $F_{5}$
      when (a) $F_{6} = \{u_{6},u_{7},u_{8}\}$ and (b)
      $F_{5} = \{u_{4},u_{7},u_{8}\}$.}
    \label{fig:clm5}
  \end{figure}

  By Claim~\ref{clm:5}, we have $\deg_{T} t_{6} = 2$, and thus
  $\dist_{T}(t, t_{1}) \neq 5$.
  Let $t_{7} (\neq t_{5})$ be the vertex adjacent to $t_{6}$, and let
  $F_{7}$ be the triangle corresponding to $t_{7}$.  Let $u_{9}$ be
  the vertex in $F_{7}$ that is not in $F_{6}$.

  \begin{claim}
    \label{clm:6}
    $V(F_{7}) = \{u_{6}, u_{8}, u_{9}\}$ and $\deg_{T} t_{7} = 3$.
  \end{claim}

  \begin{figure}[tbp]
    \centering
    \includegraphics[width=\textwidth]{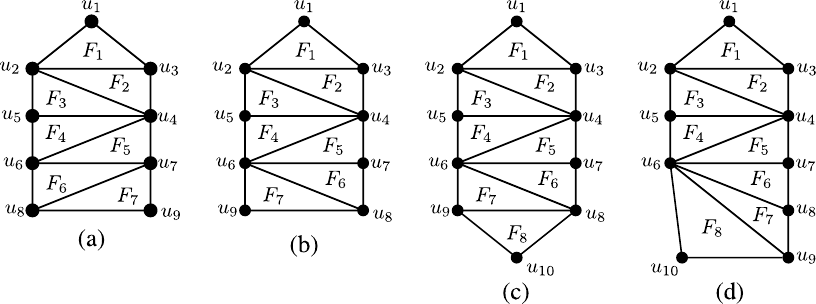}
    \caption{Claim~\ref{clm:6}.  Possible triangles adjacent to $F_{6}$.}
    \label{fig:clm6}
  \end{figure}

  \begin{proof}[Proof of Claim~\ref{clm:6}]
    By Claim~\ref{clm:5}, we have
    $V(F_{6}) = \{u_{6}, u_{7}, u_{8}\}$.  Thus $V(F_{7})$ is either
    $\{u_{7}, u_{8}, u_{9}\}$ or $\{u_{6}, u_{8}, u_{9}\}$.

    Suppose to the contrary that $V(F_{7}) = \{u_{7}, u_{8}, u_{9}\}$
    (See Fig.~\ref{fig:clm6}(a)).  Let
    $G' = G - \{u_{1}, u_{2}, u_{3}, u_{4}, u_{5}, u_{6}\}$.  $G'$
    has $n-6$ vertices and $k' \leq k$ bad vertices.  Thus, a minimum
    double dominating set $S'$ of $G'$ satisfies
    $|S'| \leq (n'+k')/2 = (n+k-6)/2$.
    The set $S = S' \cup \{u_{2}, u_{3}, u_{6}\}$ is a
    double dominating set of $G$ and satisfies
    $\gamma_{\times 2}(G) \leq |S| = |S'| + 3 = (n+k)/2$, a contradiction.

    Hence we have $V(F_{7}) = \{u_{6}, u_{8}, u_{9}\}$ (See
    Fig.~\ref{fig:clm6}(b)).

    Next we show that $\deg_{T} t_{7} = 3$.  First assume that
    $\deg_{T} t_{7} = 1$.
    In this case, $G$ is isomorphic to Fig.~\ref{fig:clm6}(b), and
    thus $\gamma_{\times 2}(G) = 4 \leq (n+k)/2$ since
    $\{u_{2}, u_{3}, u_{6}, u_{8}\}$ is a minimum double dominating
    set.

    Next assume that $\deg_{T} t_{7} = 2$.
    Let $t_{8} (\neq t_{6})$ be the vertex adjacent to $t_{7}$, and
    let $F_{8}$ be the triangle corresponding to $t_{8}$.
    Let $u_{10}$ be the vertex in $F_{8}$ that is not in $F_{7}$.
    In this case, $V(F_{8})$ is either $\{u_{8}, u_{9}, u_{10}\}$
    (Fig.\ref{fig:clm6}(c)) or $\{u_{6}, u_{9}, u_{10}\}$
    (Fig.\ref{fig:clm6}(d)).
    For both cases, let
    $G' = G - \{u_{1}, u_{2}, u_{3}, u_{4}, u_{5}, u_{7}\}$.
    $G'$ has $n-6$ vertices and $k' \leq k$ bad vertices.
    Thus, a minimum double dominating set $S'$ of $G'$ satisfies
    $|S'| \leq (n'+k')/2 = (n+k-6)/2$.

    If $V(F_{8}) = \{u_{8}, u_{9}, u_{10}\}$, then $S'$ contains
    $u_{8}$ and $u_{9}$ since $\deg_{G'} u_{6} = 2$, thus
    $S = S' \cup \{u_{2}, u_{3}, u_{4}\}$ is a double dominating set
    of $G$ and satisfies
    $\gamma_{\times 2}(G) \leq |S| = |S'| + 3 = (n+k)/2$, a contradiction.
    If $V(F_{8}) = \{u_{6}, u_{9}, u_{10}\}$, then $S'$ contains
    $u_{6}$ and $u_{9}$ since $\deg_{G'} u_{8} = 2$, thus
    $S = S' \cup \{u_{2}, u_{3}, u_{7}\}$ is a double dominating set
    of $G$ and satisfies
    $\gamma_{\times 2}(G) \leq |S| = |S'| + 3 = (n+k)/2$, a
    contradiction.

    From the above discussion, we have $\deg_{T} t_{7} \neq 1$ and
    $\deg_{T} t_{7} \neq 2$.
    Hence we obtain $\deg_{T} t_{7} = 3$.
  \end{proof}

  If $G$ has $n \geq 9$ vertices and there is no vertex of degree 3 in
  $T$, then it contradicts Claim~\ref{clm:6}.
  Hence we obtain the property~(1).

  If $t$ is the nearest vertex of degree 3 from $t_{1}$ in $T$, then
  $\dist_{T}(t_{1}, t) \leq 6$.  If $t = t_{7}$, then
  $\dist_{T}(t_{1}, t) = 6$ and the property~(2) holds.  If
  $t=t_{6}$, then $F_{7}$ is a internal triangle, and
  $\{u_{1}, u_{2}, \dots,u_{9}\}$ induces the subgraph in
  Fig.~\ref{fig:lem}(d).
  Hence the property~(4) holds.
\end{proof}

By Lemma~\ref{lem:dist}, there is a vertex of degree 3 in $T$.
Let $s_{1}$ and $t_{1}$ be vertices of degree 1 of $T$ such that they
have the common nearest vertex $t$ of degree 3.
Let $\dist_{T}(s_{1}, t) = d_{s}$ and $\dist_{T}(t_{1}, t) = d_{t}$.
Without loss of generality, we may assume that $d_{s} \leq d_{t}$.
By Lemma~\ref{lem:dist}, $d_{s}, d_{t} \in \{1, 2, 4, 6\}$, and the
subgraphs induced by the vertices on the triangles that are
corresponding to the paths from $s_{1}$ or $t_{1}$ to $t$ in $T$ are
isomorphic to the graphs in Fig.~\ref{fig:lem}.

In the following, the subgraph induced by the triangles that are
corresponding to the vertices in the path from $s_{1}$ to $t$ has
vertices $u_{1}, u_{2}, \dots$.
Similarly, the subgraph induced by the triangles that are
corresponding to the vertices in the path from $t_{1}$ to $t$ has
vertices $v_{1}, v_{2}, \dots$.

We then consider the four cases when (1) $d_{s} = 1$, (2) $d_{s} = 2$,
(3) $d_{s} = 4$, and (4) $d_{s} = 6$.

\begin{figure}[tbp]
  \centering
  \includegraphics{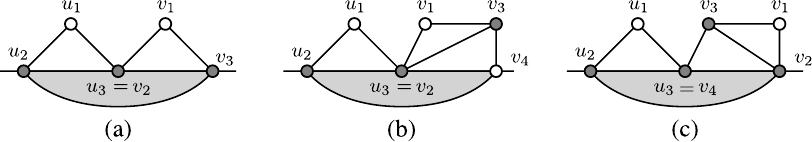}
  \caption{Case~1 and Case~2.  Possible situations when $d_s = 1$ and
    $d_{t} = 1$ or 2.  The shaded triangle is an internal triangle
    corresponding to $t$.  The gray vertices are in a double
    dominating set $S$.}
  \label{fig:ds1-2}
\end{figure}

\vspace{1em}
\noindent
(Case~1) When $d_{s} = 1$.
We then consider the four subcases.

(Case~1-1) $d_{t} = 1$.
See Fig.~\ref{fig:ds1-2}(a).

Let $G' = G - \{u_{1}, v_{1}\}$.  $G'$ has $n-2$ vertices and
$k' \leq k-1$ bad vertices.  Thus, a minimum double dominating set
$S'$ of $G'$ satisfies
$|S'| \leq (n'+k')/2 \leq (n+k-2)/2$,
and $S'$ includes $u_{2}$ and $v_{3}$ since $\deg_{G'} u_{3} = 2$.
The set $S = S' \cup \{u_{3}\}$ is a double dominating set of $G$ and
satisfies
$\gamma_{\times 2}(G) \leq |S| = |S'| + 1 \leq (n+k)/2$, a
contradiction.

\vspace{1em}
(Case~1-2) $d_{t} = 2$.
See Fig.~\ref{fig:ds1-2}(b) and (c).

For Fig.~\ref{fig:ds1-2}(b), let $G' = G - \{v_{1}, u_{3}\}$.
Note that $v_{1}$ is a bad vertex of $G$, and $u_{1}$ is bad in $G'$.
Thus $G'$ has $n-2$ vertices and $k' = k$ bad vertices, and a
minimum double dominating set $S'$ of $G'$ satisfies
$|S'| \leq (n'+k')/2 \leq (n+k-2)/2$,
and $S'$ includes $u_{2}$ and $u_{3}$ since $\deg_{G'} u_{1} = 2$.
The set $S = S' \cup \{v_{3}\}$ is a double dominating set of $G$ and
satisfies
$\gamma_{\times 2}(G) \leq |S| = |S'| + 1 \leq (n+k)/2$, a
contradiction.

For Fig.~\ref{fig:ds1-2}(c), let $G' = G - \{u_{1}, v_{1}, u_{3}\}$.
If $v_{1}$ is bad (resp.~good) in $G$, then $u_{3}$ is bad
(resp.~good) in $G'$.  Since $u_{1}$ is a bad vertex in $G$, $G'$ has
$k' = k-1$ bad vertices.  Thus, a minimum double dominating set $S'$
of $G'$ satisfies $|S'| \leq (n'+k')/2 \leq (n+k-4)/2$, and $S'$
includes $u_{2}$ and $v_{2}$ since $\deg_{G'} u_{3} = 2$.  The set
$S = S' \cup \{u_{3}, v_{3}\}$ is a double dominating set of $G$ and
satisfies $\gamma_{\times 2}(G) \leq |S| = |S'| + 2 \leq (n+k)/2$, a
contradiction.

\begin{figure}[tbp]
  \centering
  \includegraphics{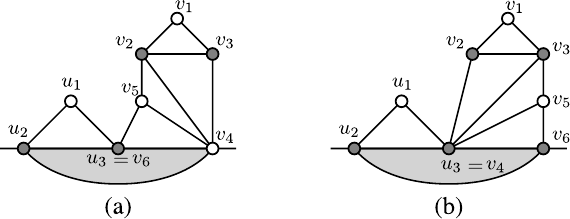}
  \caption{Case~1-3.  Possible situations when $d_s = 1$ and
    $d_{t} = 4$.  The shaded triangle is an internal triangle
    corresponding to $t$.  The gray vertices are in a double
    dominating set $S$.}
  \label{fig:ds1-4}
\end{figure}

\vspace{1em}
(Case~1-3) $d_{t} = 4$.
See Fig.~\ref{fig:ds1-4}.

For Fig.~\ref{fig:ds1-4}(a), let $G' = G - \{v_{1}, v_{2}, v_{3},
v_{5}\}$.
Then $G'$ has $n-4$ vertices and $k' \leq k-1$ bad vertices, and a
minimum double dominating set $S'$ of $G'$ satisfies
$|S'| \leq (n'+k')/2 \leq (n+k-5)/2$,
and $S'$ includes $u_{2}$ and $u_{3}$ since $\deg_{G'} u_{1} = 2$.
The set $S = S' \cup \{v_{2}, v_{3}\}$ is a double dominating set of
$G$ and satisfies
$\gamma_{\times 2}(G) \leq |S| = |S'| + 2 \leq (n+k)/2$, a
contradiction.

For Fig.~\ref{fig:ds1-4}(b), let
$G' = G - \{u_{1}, v_{1}, v_{2}, v_{3}, v_{5}\}$.  Then $G'$ has $n-5$
vertices and $k' \leq k-1$ bad vertices, and a minimum double
dominating set $S'$ of $G'$ satisfies
$|S'| \leq (n'+k')/2 \leq (n+k-6)/2$, and $S'$ includes $u_{2}$ and
$v_{6}$ since $\deg_{G'} u_{3} = 2$.
The set $S = S' \cup \{u_{3}, v_{2}, v_{3}\}$ is a double dominating
set of $G$ and satisfies
$\gamma_{\times 2}(G) \leq |S| = |S'| + 3 \leq (n+k)/2$, a
contradiction.

\begin{figure}[tbp]
  \centering
  \includegraphics{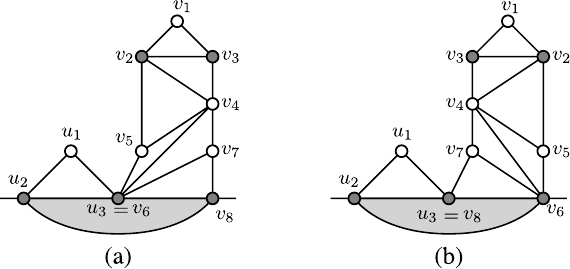}
  \caption{Case~1-4.  Possible situations when $d_s = 1$ and
    $d_{t} = 6$.  The shaded triangle is an internal triangle
    corresponding to $t$.  The gray vertices are in a double
    dominating set $S$.}
  \label{fig:ds1-6}
\end{figure}

\vspace{1em}
(Case~1-4) $d_{t} = 6$.
See Fig.~\ref{fig:ds1-6}(a) and (b).

For both cases, let $G' = G - \{v_{1}, v_{2}, v_{3}, v_{4}, v_{5}\}$.
Note that $u_{1}$ is bad in $G$ and is good in $G'$.  Then $G'$ has
$n-5$ vertices and $k' \leq k-1$ bad vertices, and a minimum double
dominating set $S'$ of $G'$ satisfies
$|S'| \leq (n'+k')/2 \leq (n+k-6)/2$.
For Fig.~\ref{fig:ds1-6}(a), $S'$ includes $u_{2}, u_{3}$ and $v_{8}$
since $\deg_{G'} u_{1} = \deg_{G'} v_{7} = 2$.
The set $S = S' \cup \{v_{2}, v_{3}\}$ is a double dominating set of
$G$ and satisfies
$\gamma_{\times 2}(G) \leq |S| = |S'| + 2 \leq (n+k)/2$, a
contradiction.
For Fig.~\ref{fig:ds1-6}(b), similarly $S'$ includes $u_{2}, u_{3}$
and $v_{6}$.
The set $S = S' \cup \{v_{2}, v_{3}\}$ is a double dominating set of
$G$ and satisfies
$\gamma_{\times 2}(G) \leq |S| = |S'| + 2 \leq (n+k)/2$, a
contradiction.

%%%%%%%%%%%%%%%%%%%%%%%%%%%%%%%%%%%%%%%%%%%%%%%%%%%
\vspace{1em}
\noindent
(Case~2) When $d_{s} = 2$.
We then consider the three subcases.

\begin{figure}[tbp]
  \centering
  \includegraphics{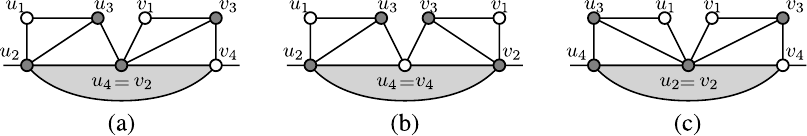}
  \caption{Case 2-1.  Possible situations when $d_s = d_{t} = 2$.  The
    shaded triangle is an internal triangles corresponding to $t$.
    The gray vertices are in a double dominating set $S$.}
  \label{fig:ds_22}
\end{figure}

\vspace{1em}
(Case~2-1) $d_{t} = 2$.  See Fig.~\ref{fig:ds_22}.

For Fig.~\ref{fig:ds_22}(a), let $G' = G - \{u_{1}, v_{1}, v_{3}\}$.
$G'$ has $n-3$ vertices and $k' \leq k-1$ bad vertices.
Thus, a minimum double dominating set $S'$ of $G'$ satisfies
$|S'| \leq (n'+k')/2 \leq (n+k-4)/2$, and $S'$ includes $u_{2}$ and
$u_{4}$ since $\deg_{G'} u_{3} = 2$.
The set $S = S' \cup \{u_{3}, v_{3}\}$ is a double dominating set of
$G$ and satisfies
$\gamma_{\times 2}(G) \leq |S| = |S'| + 2 \leq (n+k)/2$, a
contradiction.

For Fig.~\ref{fig:ds_22}(b), let
$G' = G - \{u_{1}, u_{3}, v_{1}, v_{3}\}$.
$G'$ has $n-4$ vertices and $k' \leq k-1$ bad vertices.
Thus, a minimum double dominating set $S'$ of $G'$ satisfies
$|S'| \leq (n'+k')/2 \leq (n+k-5)/2$, and $S'$ includes $u_{2}$ and
$v_{2}$ since $\deg_{G'} u_{4} = 2$.
The set $S = S' \cup \{u_{3}, v_{3}\}$ is a double dominating set of
$G$ and satisfies
$\gamma_{\times 2}(G) \leq |S| = |S'| + 2 \leq (n+k)/2$, a
contradiction.

For Fig.~\ref{fig:ds_22}(c), let $G' = G - \{u_{1}, v_{1}, v_{3}\}$.
$G'$ has $n-3$ vertices and $k' \leq k-1$ bad vertices.
Thus, a minimum double dominating set $S'$ of $G'$ satisfies
$|S'| \leq (n'+k')/2 \leq (n+k-4)/2$, and $S'$ includes $u_{2}$ and
$u_{4}$ since $\deg_{G'} u_{3} = 2$.
The set $S = S' \cup \{u_{3}, v_{3}\}$ is a double dominating set of
$G$ and satisfies
$\gamma_{\times 2}(G) \leq |S| = |S'| + 2 \leq (n+k)/2$, a
contradiction.

\begin{figure}[tbp]
  \centering
  \includegraphics[width=\textwidth]{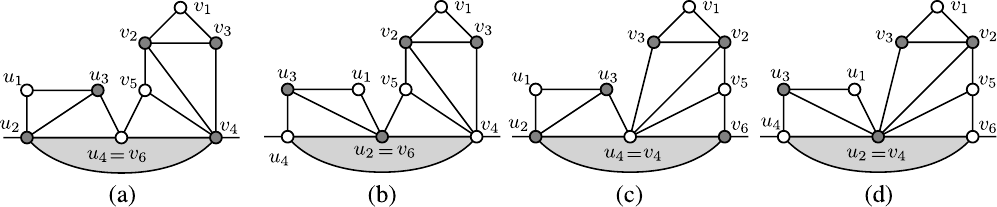}
  \caption{Case~2-2.  Possible situations when $d_s = 2$ and
    $d_{t} = 4$.  The shaded triangle is an internal triangle
    corresponding to $t$.  The gray vertices are in a double
    dominating set $S$.}
  \label{fig:ds_24}
\end{figure}

\vspace{1em}
(Case~2-2) $d_{t} = 4$.
See Fig.~\ref{fig:ds_24}.

For Fig.~\ref{fig:ds_24}(a), let
$G' = G - \{u_{1}, u_{3}, v_{1}, v_{2}, v_{3}, v_{5}\}$.
$G'$ has $n-6$ vertices and $k' \leq k-1$ bad vertices.
Thus, a minimum double dominating set $S'$ of $G'$ satisfies
$|S'| \leq (n'+k')/2 \leq (n+k-7)/2$, and $S'$ includes $u_{2}$ and
$v_{4}$ since $\deg_{G'} u_{4} = 2$.
The set $S = S' \cup \{u_{3}, v_{2}, v_{3}\}$ is a double dominating
set of $G$ and satisfies
$\gamma_{\times 2}(G) \leq |S| = |S'| + 3 \leq (n+k)/2$, a
contradiction.

For Fig.~\ref{fig:ds_24}(b), let
$G' = G - \{v_{1}, v_{2}, v_{3}, v_{5}\}$.
$G'$ has $n-4$ vertices and $k' \leq k-1$ bad vertices.
Thus, a minimum double dominating set $S'$ of $G'$ satisfies
$|S'| \leq (n'+k')/2 \leq (n+k-5)/2$, and $S'$ includes $u_{2}$ and
$u_{3}$ since $\deg_{G'} u_{1} = 2$.
The set $S = S' \cup \{v_{2}, v_{3}\}$ is a double dominating set of
$G$ and satisfies
$\gamma_{\times 2}(G) \leq |S| = |S'| + 2 \leq (n+k)/2$, a
contradiction.

For Fig.~\ref{fig:ds_24}(c), let
$G' = G - \{u_{1}, u_{3}, v_{1}, v_{2}, v_{3}, v_{5}\}$.
$G'$ has $n-6$ vertices and $k' \leq k-1$ bad vertices.
Thus, a minimum double dominating set $S'$ of $G'$ satisfies
$|S'| \leq (n'+k')/2 \leq (n+k-7)/2$, and $S'$ includes $u_{2}$ and
$v_{6}$ since $\deg_{G'} u_{4} = 2$.
The set $S = S' \cup \{u_{3}, v_{2}, v_{3}\}$ is a double dominating
set of $G$ and satisfies
$\gamma_{\times 2}(G) \leq |S| = |S'| + 3 \leq (n+k)/2$, a
contradiction.

For Fig.~\ref{fig:ds_24}(d), let
$G' = G - \{v_{1}, v_{2}, v_{3}, v_{5}\}$.
$G'$ has $n-4$ vertices and $k' \leq k-1$ bad vertices.
Thus, a minimum double dominating set $S'$ of $G'$ satisfies
$|S'| \leq (n'+k')/2 \leq (n+k-5)/2$, and $S'$ includes $u_{2}$ and
$u_{3}$ since $\deg_{G'} u_{1} = 2$.
The set $S = S' \cup \{v_{2}, v_{3}\}$ is a double dominating set of
$G$ and satisfies
$\gamma_{\times 2}(G) \leq |S| = |S'| + 2 \leq (n+k)/2$, a
contradiction.

\begin{figure}[tbp]
  \centering
  \includegraphics[width=\textwidth]{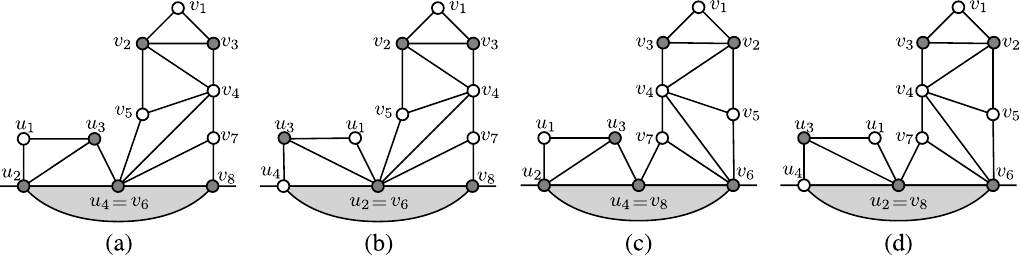}
  \caption{Case~2-3.  Possible situations when $d_s = 2$ and
    $d_{t} = 6$.  The shaded triangle is an internal triangle
    corresponding to $t$.  The gray vertices are in a double
    dominating set $S$.}
  \label{fig:ds_26}
\end{figure}

\vspace{1em}
(Case~2-3) $d_{t} = 6$.
See Fig.~\ref{fig:ds_26}.

For Fig.~\ref{fig:ds_26}(a), let
$G' = G - \{u_{1}, u_{3}, v_{1}, v_{2}, v_{3}, v_{4}, v_{5},
v_{7}\}$.  $G'$ has $n-8$ vertices and $k' \leq k-1$ bad vertices.
Thus, a minimum double dominating set $S'$ of $G'$ satisfies
$|S'| \leq (n'+k')/2 \leq (n+k-9)/2$,
and $S'$ includes $u_{2}$ and $v_{8}$ since $\deg_{G'} u_{4} = 2$.
The set $S = S' \cup \{u_{3}, u_{4}, v_{2}, v_{3}\}$ is a double
dominating set of $G$ and satisfies
$\gamma_{\times 2}(G) \leq |S| = |S'| + 4 \leq (n+k)/2$, a
contradiction.

For Fig.~\ref{fig:ds_26}(b), let
$G' = G - \{v_{1}, v_{2}, v_{3}, v_{4}, v_{5}, v_{7}\}$.  $G'$ has
$n-6$ vertices and $k' \leq k-1$ bad vertices.  Thus, a minimum double
dominating set $S'$ of $G'$ satisfies
$|S'| \leq (n'+k')/2 \leq (n+k-7)/2$,
and $S'$ includes $u_{2}$ and $u_{3}$ since $\deg_{G'} u_{1} = 2$.
The set $S = S' \cup \{v_{2}, v_{3}, v_{8}\}$ is a double dominating
set of $G$ and satisfies
$\gamma_{\times 2}(G) \leq |S| = |S'| + 3 \leq (n+k)/2$, a
contradiction.

For Fig.~\ref{fig:ds_26}(c), let
$G' = G - \{u_{1}, u_{3}, v_{1}, v_{2}, v_{3}, v_{4}, v_{5}, v_{7}\}$.
$G'$ has $n-8$ vertices and $k' \leq k-1$ bad vertices.
Thus, a minimum double dominating set $S'$ of $G'$ satisfies
$|S'| \leq (n'+k')/2 \leq (n+k-9)/2$, and $S'$ includes $u_{2}$ and
$v_{6}$ since $\deg_{G'} u_{4} = 2$.
The set $S = S' \cup \{u_{3}, u_{4}, v_{2}, v_{3}\}$ is a double
dominating set of $G$ and satisfies
$\gamma_{\times 2}(G) \leq |S| = |S'| + 4 \leq (n+k)/2$, a
contradiction.

For Fig.~\ref{fig:ds_26}(d), let
$G' = G - \{v_{1}, v_{2}, v_{3}, v_{4}, v_{5}, v_{7}\}$.
$G'$ has $n-6$ vertices and $k' \leq k-1$ bad vertices.
Thus, a minimum double dominating set $S'$ of $G'$ satisfies
$|S'| \leq (n'+k')/2 \leq (n+k-7)/2$, and $S'$ includes $u_{2}$ and
$u_{3}$ since $\deg_{G'} u_{1} = 2$.
The set $S = S' \cup \{v_{2}, v_{3}, v_{6}\}$ is a double dominating
set of $G$ and satisfies
$\gamma_{\times 2}(G) \leq |S| = |S'| + 3 \leq (n+k)/2$, a
contradiction.

%%%%%%%%%%%%%%%%%%%%%%%%%%%%%%%%%%%%%%%%%%%%%%%%%%%
\vspace{1em}
\noindent
(Case~3) When $d_{s} = 4$.  We then consider the two cases.

\begin{figure}[tbp]
  \centering
  \includegraphics[]{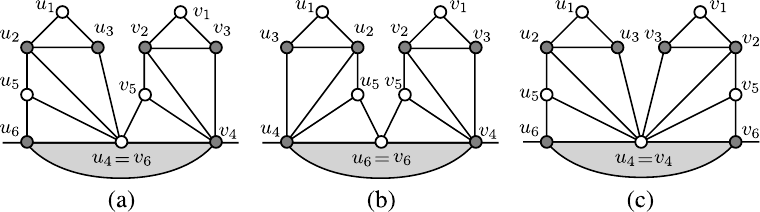}
  \caption{Case~3-1.  Possible situations when $d_s = 4$ and
    $d_{t} = 4$.  The shaded triangle is an internal triangle
    corresponding to $t$.  The gray vertices are in a double
    dominating set $S$.}
  \label{fig:ds_44}
\end{figure}

\vspace{1em}
(Case~3-1) $d_{t} = 4$.
See Fig.~\ref{fig:ds_44}.

For Fig.~\ref{fig:ds_44}(a), let
$G' = G - \{u_{1}, u_{2}, u_{3}, u_{5}, v_{1}, v_{2}, v_{3}, v_{5}\}$.
$G'$ has $n-8$ vertices and $k' \leq k-1$ bad vertices.  Thus, a
minimum double dominating set $S'$ of $G'$ satisfies
$|S'| \leq (n'+k')/2 \leq (n+k-9)/2$,
and $S'$ includes $u_{6}$ and $v_{4}$ since $\deg_{G'} u_{4} = 2$.
The set $S = S' \cup \{u_{2}, u_{3}, v_{2}, v_{3}\}$ is a double
dominating set of $G$ and satisfies
$\gamma_{\times 2}(G) \leq |S| = |S'| + 4 \leq (n+k)/2$, a
contradiction.

For Fig.~\ref{fig:ds_44}(b), let
$G' = G - \{u_{1}, u_{2}, u_{3}, u_{5}, v_{1}, v_{2}, v_{3}, v_{5}\}$.
$G'$ has $n-8$ vertices and $k' \leq k-1$ bad vertices.  Thus, a
minimum double dominating set $S'$ of $G'$ satisfies
$|S'| \leq (n'+k')/2 \leq (n+k-9)/2$,
and $S'$ includes $u_{4}$ and $v_{4}$ since $\deg_{G'} u_{6} = 2$.
The set $S = S' \cup \{u_{2}, u_{3}, v_{2}, v_{3}\}$ is a double
dominating set of $G$ and satisfies
$\gamma_{\times 2}(G) \leq |S| = |S'| + 4 \leq (n+k)/2$, a
contradiction.

For Fig.~\ref{fig:ds_44}(c), let
$G' = G - \{u_{1}, u_{2}, u_{3}, u_{5}, v_{1}, v_{2}, v_{3}, v_{5}\}$.
$G'$ has $n-8$ vertices and $k' \leq k-1$ bad vertices.  Thus, a
minimum double dominating set $S'$ of $G'$ satisfies
$|S'| \leq (n'+k')/2 \leq (n+k-9)/2$,
and $S'$ includes $u_{6}$ and $v_{6}$ since $\deg_{G'} u_{4} = 2$.
The set $S = S' \cup \{u_{2}, u_{3}, v_{2}, v_{3}\}$ is a double
dominating set of $G$ and satisfies
$\gamma_{\times 2}(G) \leq |S| = |S'| + 4 \leq (n+k)/2$, a
contradiction.

\begin{figure}[tbp]
  \centering
  \includegraphics[width=\textwidth]{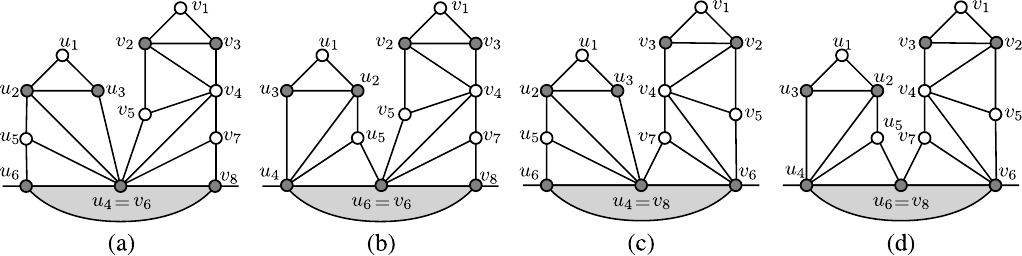}
  \caption{Case~3-2.  Possible situations when $d_s = 4$ and
    $d_{t} = 6$.  The shaded triangle is an internal triangle
    corresponding to $t$.  The gray vertices are in a double
    dominating set $S$.}
  \label{fig:ds_46}
\end{figure}

\vspace{1em}
(Case~3-2) $d_{t} = 6$.
See Fig.~\ref{fig:ds_46}.

For each case in Fig.~\ref{fig:ds_46}, let
$G' = G - \{u_{1}, u_{2}, u_{3}, v_{1}, v_{2}, v_{3}, v_{4}, v_{5}\}$.
$G'$ has $n-8$ vertices and $k' \leq k-1$ bad vertices since $u_{5}$
is a good vertex in $G'$.  Thus, a minimum double dominating set $S'$
of $G'$ satisfies
$|S'| \leq (n'+k')/2 \leq (n+k-9)/2$,
and $S'$ includes $u_{4}, u_{6}, v_{6}$ and $v_{8}$ since
$\deg_{G'} u_{5} = \deg_{G'} v_{7} = 2$.  The set
$S = S' \cup \{u_{2}, u_{3}, v_{2}, v_{3}\}$ is a double dominating
set of $G$ and satisfies
$\gamma_{\times 2}(G) \leq |S| = |S'| + 4 \leq (n+k)/2$, a
contradiction.

%%%%%%%%%%%%%%%%%%%%%%%%%%%%%%%%%%%%%%%%%%%%%%%%%%%
\begin{figure}[tbp]
  \centering
  \includegraphics{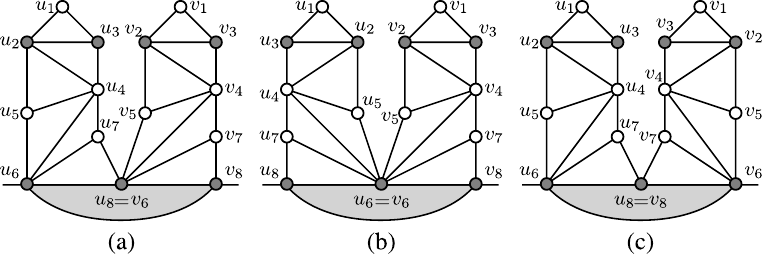}
  \caption{Case~4.  Possible situations when $d_s = 6$ and
    $d_{t} = 6$.  The shaded triangle is an internal triangle
    corresponding to $t$.  The gray vertices are in a double
    dominating set $S$.}
  \label{fig:ds_66}
\end{figure}

\vspace{1em}
\noindent
(Case~4) $d_{s} = d_{t} = 6$.
See Fig.~\ref{fig:ds_66}.

For each case in Fig.~\ref{fig:ds_66}, let
$G' = G - \{u_{1}, u_{2}, u_{3}, u_{4}, u_{5}, v_{1}, v_{2}, v_{3},
v_{4}, v_{5}\}$.  $G'$ has $n-10$ vertices and $k' \leq k$ bad
vertices.  Thus, a minimum double dominating set $S'$ of $G'$
satisfies
$|S'| \leq (n'+k')/2 \leq (n+k-10)/2$,
and $S'$ includes $u_{6}, u_{8}, v_{6}$ and $v_{8}$ since
$\deg_{G'} u_{7} = \deg_{G'} v_{7} = 2$.  The set
$S = S' \cup \{u_{2}, u_{3}, v_{2}, v_{3}\}$ is a double dominating
set of $G$ and satisfies
$\gamma_{\times 2}(G) \leq |S| = |S'| + 4 \leq (n+k)/2$, a
contradiction.

\vspace{1em}
From the above four cases, we complete the proof of
Theorem~\ref{thm:main}.

\vspace{1em}

Finally, we discuss the lower bound of the double domination number of
outerplanar graphs.
In the paper~\cite{aita24}, the author showed that the 2-domination
number of an outerplanar graph satisfies
$\gamma_{2}(G) \geq \lceil (n+2)/3 \rceil$.
The 2-dominating set $S$ of a graph $G$ is \emph{$2$-dominating set}
if, for any $v \notin S$, there exists at least two vertices adjacent
to $v$.
The \emph{2-domination number} $\gamma_{2}(G)$ is the minimum
cardinality of a 2-dominating set of $G$.
Since $\gamma_{2}(G) \leq \gamma_{\times 2}(G)$, this lower bound also
holds for the double dominating number.

To prove that this lower bound is tight, we construct infinite family
of maximal outerplanar graphs whose double domination number is equal
to the lower bound.
For any $k \geq 1$, let $G_{k}$ be a graph that has the vertex set
$V(G_{k}) = \{v_{1},v_{2},\dots,v_{3k+1}\}$, and
$E(G_{k}) = \{v_{i}v_{i+1} \mid i = 2,3,\dots,3k\} \cup \{v_{1} v_{i}
\mid i = 2,3,\dots,3k+1 \}$.

Let $S=\{v_{1}\} \cup \{v_{3i} \mid i=1,2,\dots,k\}$.
Then $S$ is a double dominating set of $G_{k}$, and $|S|=\lceil
(n+2)/3 \rceil=k+1$.
It is easy to see that there is no double dominating set $S'$ such
that $|S'| \leq k$.
Therefore, $\gamma_{\times 2}(G_{k}) = k+1$ for any $k \geq 1$.

% We show that $S$ is a minimum double dominating set of $G_{k}$.
% Assume to the contrary that there is a double dominating set $S'$ of
% $k$ vertices.
% For $j=1,2,\dots,k$, let $V_{j}=\{v_{3j-1},v_{3j},v_{3j+1}\}$.
% Then $V(G)=V_{0} \cup V_{1} \cup \dots \cup V_{k}$, where
% $V_{0}=\{v_{1}\}$.
% Since $|S'|=k$, we have $V_{j} \cap S' = \emptyset$ for some $j$.
% If $V_{j} \cap S' = \emptyset$ for some $j \geq 1$, then $S'$ is not a
% double dominating set.
% Hence $v_{1} \notin S'$.
% Since $|S'|=k$ and $S'$ is a dominating set, $|S' \cap V_{j}|=1$ for
% each $j=1,2,\dots,k$.
% However, it implies that $S'$ cannot be a total dominating set.

\section*{Acknowledgments}
\label{sec:acknowledgments}

This work was supported by JSPS KAKENHI Grant Number JP22K11898.

% \bibliographystyle{elsart-num-sort}
% \bibliography{araki}

\end{document}